\documentclass[a4paper,11pt,reqno,draft]{amsart}
\usepackage{amssymb,amsmath,array}

\title[Weyl submodules in restrictions of simple modules]
{Weyl submodules in restrictions of simple modules}

\author{Vladimir Shchigolev}

\address{
\noindent Department of Algebra, Faculty of Mathematics, Lomonosov
Moscow State University, Leninskiye Gory, Moscow, 119899, RUSSIA}

\email{shchigolev\_vladimir@yahoo.com}


\renewcommand{\phi}{\varphi}
\renewcommand{\kappa}{\varkappa}
\renewcommand{\epsilon}{\varepsilon}

\newenvironment{explicit}[1]{\par\noindent{\bf#1.} \it}{\par}

\def\F{\mathbb F}
\def\Z{\mathbb Z}
\def\C{\mathbb C}

\def\Hom{\mathop{\rm Hom}}
\def\ch{\mathop{\rm char}}
\def\rad{\mathop{\rm rad}}
\def\SL{{\rm SL}}
\def\sl{{\mathfrak sl}}

\def\coeff{\mathop{\rm coeff}\nolimits}
\def\|{\mathbin|}
\def\diag{\mathop{\rm diag}\nolimits}

\newtheorem{theorem}{Theorem}
\newtheorem{proposition}[theorem]{Proposition}
\newtheorem{lemma}[theorem]{Lemma}
\newtheorem{corollary}[theorem]{Corollary}
\newtheorem{problem}[theorem]{Problem}

\newenvironment{remark}{\refstepcounter{theorem}\par{\noindent\bf Remark \arabic{theorem}.}}{}

\renewcommand{\(}{\left(}
\renewcommand{\)}{\right)}
\def\<{\langle}
\def\>{\rangle}
\renewcommand{\le}{\leqslant}
\renewcommand{\ge}{\geqslant}
\def\={\equiv}

\def\bbar#1{\bar{\bar #1}}

\def\U{\mathcal U}

\begin{document}

\maketitle

\begin{abstract} Let $\F$ be an algebraically closed field of
characteristic $p>0$. Suppose that $\SL_{n-1}(\F)$ is naturally
embedded into $\SL_n(\F)$ (either in the top left corner or
in the bottom right corner). We prove that certain Weyl modules
over $\SL_{n-1}(\F)$ can be embedded into
the restriction $L(\omega)\!\!\downarrow_{\SL_{n-1}(\F)}$,
where $L(\omega)$ is a simple $\SL_n(\F)$-module.
This allows us to construct new primitive vectors in
$L(\omega)\!\!\downarrow_{\SL_{n-1}(\F)}$ from any primitive
vectors in the corresponding Weyl modules.
Some examples are given to show that this result actually works.

\end{abstract}

\section{Introduction}

Let $G=\SL_n(\F)$, where $\F$ is an algebraically closed
field of characteristic $p>0$ and $n\ge3$.
Consider the subgroup $G^{(q)}$ of $G$
generated by the root elements $x_\alpha(t)$, $x_{-\alpha}(t)$,
where $\alpha$ is a simple root distinct from a fixed terminal (simple)
root $\alpha_q$. It is a classical problem to describe
the structure of the restriction $L{\downarrow}_{G^{(q)}}$,
where $L$ is a simple rational $G$-module.

In this paper, we focus on primitive (with respect to $G^{(q)}$)
vectors of $L{\downarrow}_{G^{(q)}}$.
The complete combinatorial description of these vectors is an open problem
(stated in~\cite{Kleshchev_bou}),
although lately there has been some progress in this direction
(\cite{Kleshchev2},~\cite{Kleshchev_gjs11},~\cite{Shchigolev_submitted_JA_2007}).

Another problem of equal importance is the description of primitive vectors
in Weyl modules. Known methods of constructing such vectors
(\cite{Carter_Lusztig}, \cite{Carter_Payne}) and methods of constructing
primitive vectors in restrictions $L{\downarrow}_{G^{(q)}}$
(\cite{Kleshchev2}, \cite{Kleshchev_gjs11}, \cite{Shchigolev15}, \cite{Shchigolev_submitted_JA_2007})
bear some similarity (e.g. similar lowering operators),
which is still not fully understood.

The present paper contains a combinatorial condition
under which all primitive vectors (regardless of their nature)
of certain Weyl modules over $G^{(q)}$ become primitive vectors
of $L{\downarrow}_{G^{(q)}}$.
This result is proved by embedding the corresponding Weyl
modules into $L{\downarrow}_{G^{(q)}}$ (Theorem~A).
Examples~I and II show that our result actually works,
that is, produces nonzero
primitive vectors of $L{\downarrow}_{G^{(q)}}$.

We also hope that Theorem~A, will be useful for finding
new composition factors of $L{\downarrow}_{G^{(q)}}$
and lower estimates of the dimensions of the weight spaces~of~$L$.

We order the simple roots $\alpha_1,\ldots,\alpha_{n-1}$ so that
$x_{\alpha_i}(t)=E+te_{i,i+1}$.
Then $x_{\alpha_i+\cdots+\alpha_{j-1}}(t)=E+te_{i,j}$ and
$x_{-\alpha_i-\cdots-\alpha_{j-1}}(t)=E+te_{j,i}$,
where $1\le i<j\le n$.
Here and in what follows $E$ is the identity $n\times n$ matrix
and $e_{i,j}$ is the $n\times n$ matrix having $1$ in the $ij$th
position and $0$ elsewhere. The root system $\Phi$ of $G$
consists of the roots $\pm(\alpha_i+\cdots+\alpha_{j-1})$
and the positive root system $\Phi^+$ consists of the roots
$\alpha_i+\cdots+\alpha_{j-1}$, where $1\le i<j\le n$.
Let $\omega_1,\ldots,\omega_{n-1}$ denote the fundamental weights
corresponding to the roots $\alpha_1,\ldots,\alpha_{n-1}$.

In $G$, we fix the maximal torus $T$ consisting of diagonal matrices
and the Borel subgroup $B$ consisting of upper triangular
matrices.

The {\it hyperalgebra} of $G$ is constructed as follows.
Consider the following elements of $\sl_n(\C)$:
$X_{\alpha_i+\cdots+\alpha_{j-1}}=e_{i,j}$,
$X_{-\alpha_i-\cdots-\alpha_{j-1}}=e_{j,i}$, where $1\le i<j\le n$,
and $H_{\alpha_i}=e_{i,i}-e_{i+1,i+1}$, where $1\le i<n$.
Following~\cite[Theorem~2]{Steinberg_eng}, we denote by $\U_\Z$
the subring of the universal enveloping algebra of $\sl_n(\C)$
generated by divided powers $X_\alpha^m/m!$, where $\alpha\in\Phi$ and
$m\in\Z^+$ (the set of nonnegative integers).
The hyperalgebra of $G$ is the tensor product $\U:=\U_\Z\otimes_\Z\F$.
Elements $X_{\alpha,m}:=(X_\alpha^m/m!)\otimes1_\F$
generate $\U$ as an $\F$-algebra.

Every rational $G$-module $V$ can be made into a $\U$-module by
the rule
\begin{equation}\label{equation:-1}
x_\alpha(t)v=\sum\nolimits_{m=0}^{+\infty}t^mX_{\alpha,m}v.
\end{equation}
We also need the
elements $H_{\alpha_i,m}=\binom{H_{\alpha_i}}m\otimes1_\F$. It is easy
to show that these elements actually belong to $\U$
(e.g.,~\cite[Corollary to Lemma~5]{Steinberg_eng}).
We shall often abbreviate $X_\alpha:=X_{\alpha,1}$
and $H_{\alpha_i}:=H_{\alpha_i,1}$ if this notation does
not cause confusion.

For any integers $q_1,\ldots,q_m\in\{1,\ldots,n-1\}$,
we denote by $G^{(q_1,\ldots,q_m)}$
the subgroup of $G$ generated by the root elements $x_{\alpha_i}(t)$, $x_{-\alpha_i}(t)$
with $i\in\{1,{\ldots},n-1\}\setminus\{q_1,{\dots},q_m\}$.
Note that $G^{(q_1,\ldots,q_m)}$ is the universal Chevalley group
with root system
$\Phi\cap\sum_{i\in\{1,\ldots,n-1\}\setminus\{q_1,\ldots,q_m\}}\Z\alpha_i$
(\cite[Theorem~27.3]{Humphreys_eng}).

In $G^{(q_1,\ldots,q_m)}$, we fix the maximal torus
$T^{(q_1,\ldots,q_m)}$ generated by the elements
$h_{\alpha_i}(t)=\diag(1,\ldots,1,t,t^{-1},1,\ldots,1)$,
where $t\in\F^*$ is at the $i$th position and $i\in\{1,{\ldots},n-1\}\setminus\{q_1,{\dots},q_m\}$,
and the Borel subgroup generated by $T^{(q_1,\ldots,q_m)}$ and the root
elements $x_{\alpha}(t)$ with
$\alpha\in\Phi^{(q_1,\ldots,q_m)}\cap\Phi^+$.

We denote by $X(T)$ the set of $T$-weights and by $X^+(T)$ the set of dominant $T$-weights.
For any $\omega\in X^+(T)$, we denote by $L(\omega)$ and $\Delta(\omega)$
the simple rational $G$-module with highest
weight $\omega$ and the Weyl $G$-module with highest weight $\omega$
respectively. We fix nonzero vectors $v^+_\omega$ and $e^+_\omega$ of
$L(\omega)$ and $\Delta(\omega)$ respectively having weight~$\omega$.
Similar notations will be used for subtori $T^{(q_1,\ldots,q_m)}$.
We shall often omit the prefix before the word ``weight''
if it is clear which torus we mean.

The terminal roots of $\Phi$ are $\alpha_1$ and $\alpha_{n-1}$.
Thus $q=1$ or $q=n-1$. For any weight $\kappa\in X(T)$, we denote
by $\bar\kappa$ and $\bbar\kappa$ the restrictions of $\kappa$ to
$T^{(1)}$ and $T^{(n-1)}$ respectively. The main results of the
present paper are as~follows.

\begin{explicit}{Theorem~A}
 Let $G=\SL_n(\F)$, $\omega\in X^+(T)$ and $k=0,\ldots,p-1$.
{\leftmargini=22pt
\begin{enumerate}
\itemsep=4pt
    \item\label{condition:G1} If $\<\omega,\alpha_1\>-l\not\equiv0\pmod p$ for any $l=0,\ldots,k-1$ and
          there is \linebreak$m=0,\ldots,k$ such that $\Delta(\bar\omega+m\bar\omega_2)$ is simple,
          then the $G^{(1)}$-submodule of $L(\omega)$ generated by $X_{-\alpha_1,k}v^+_\omega$
          is isomorphic to $\Delta(\bar\omega+k\bar\omega_2)$.
    \item\label{condition:Gn-1} If $\<\omega,\alpha_{n-1}\>-l\not\equiv0\pmod p$ for any $l=0,\ldots,k-1$ and
          there is \linebreak$m=0,\ldots,k$ such that $\Delta(\bbar\omega+m\bbar\omega_{n-2})$ is simple,
          then the $G^{(n-1)}$-sub\-module of $L(\omega)$ generated by $X_{-\alpha_{n-1},k}v^+_\omega$
          is isomorphic to $\Delta(\bbar\omega+k\bbar\omega_{n-2})$.
\end{enumerate}}
\end{explicit}

\begin{explicit}{Theorem~B}
\!\! \!\! \!\! \!\! \!\! \!\! Let $G{=}\SL_n(\F)$,
$\omega{\in}X^+(T)$, $k{=}0,{\ldots},p{-}1$ and $q{=}1$ or
$q{=}n{-}1$. The $G^{(q)}$-sub\-mo\-dule of $L(\omega)$ generated
by $X_{-\alpha_q,k}v^+_\omega$ is isomorphic to a Weyl module if
and only if $\<\omega,\alpha_q\>-l\not\equiv0\pmod p$ for any
$l=0,\ldots,k-1$ and { \leftmargini=22pt
\begin{enumerate}
\itemsep=4pt
    \item\label{condition:case:q=1} any nonzero primitive vector of $\Delta(\bar\omega+k\bar\omega_2)$ has weight
          $\bar\omega+k\bar\omega_2-b_2\bar\alpha_2-\cdots-b_{n-1}\bar\alpha_{n-1}$ with
          $k\ge b_2\ge\cdots\ge b_{n-1}\ge0$ in the case $q=1${\rm;}
    \item\label{condition:case:q=n-1} any nonzero primitive vector of $\Delta(\bbar\omega+k\bbar\omega_{n-2})$ has weight
           $\bbar\omega+k\bbar\omega_{n-2}-b_1\bbar\alpha_1-\cdots-b_{n-2}\bbar\alpha_{n-2}$
           with $0\le b_1\le\cdots\le b_{n-2}\le k$ in the case $q=n-1$.
\end{enumerate}}
\noindent More precisely, we have
$KG^{(1)}X_{-\alpha_1,k}v^+_\omega\cong\Delta(\bar\omega+k\bar\omega_2)$
in case~\ref{condition:case:q=1} and
$KG^{(n-1)}X_{-\alpha_{n-1},k}v^+_\omega\cong\Delta(\bbar\omega+k\bbar\omega_{n-2})$
in case~\ref{condition:case:q=n-1}.
\end{explicit}

Theorem~A can be viewed as a special case of the following more general problem
(valid for an arbitrary semisimple group $G$) stated by Irina Suprunenko:

\begin{problem}\label{problem:1}
Let $\alpha_q$ be a terminal root of the Dynkin diagram of $\Phi$
and $k=0,\ldots,p-1$. Describe the weights $\omega\in X^+(T)$
such that the $G^{(q)}$-submodule of the simple module $L(\omega)$
generated by $X_{-\alpha_q,k}v^+_\omega$ is isomorphic to a Weyl module.
\end{problem}

Theorem~B solves this problem for $G=\SL_n(\F)$ in terms of the Hom-spaces between Weyl modules
and is a more refined version of Theorem A giving a necessary and sufficient
condition for $X_{-\alpha_q,k} v_\omega^+$ to generate a Weyl module.

Theorem~A can easily be used in practice by
virtue of the following \linebreak irreducibility criterion of Weyl modules
over groups of type $A_{n-1}$ proved by J.C.Jantzen.

\begin{proposition}[\mbox{\cite[II.8.21]{Jantzen2}}]
The Weyl module $\Delta(\omega)$ is simple if and only if
for each $\alpha\in\Phi^+$ the following is satisfied:
Write $\<\omega+\rho,\alpha\>=ap^s+bp^{s+1}$,
where $a,b,s\in\Z^+$, $0<a<p$ and
$\rho$ is half the sum of the positive roots of $\Phi$.
Then there have to be $\beta_0,\beta_1,\ldots,\beta_b\in\Phi^+$
with $\<\omega+\rho,\beta_i\>=p^{s+1}$ for $1\le i\le b$ and $\<\omega+\rho,\beta_0\>=ap^s$,
with $\alpha=\sum_{i=0}^b\beta_i$ and with $\alpha-\beta_0\in\Phi\cup\{0\}$.
\end{proposition}

{\bf Example I.} Let
$G=\SL_3(\F)$
and $\omega=a_1\omega_1+a_2\omega_2$
be a dominant weight such that $a_1,a_2<p$ and $a_1+a_2\ge p+b$,
where $b=0,\ldots,p-2$.
We put $k:=p+b-a_2$. Note that for any $l=0,\ldots,k-1$,
we have $0<a_1-l<p$ and thus $\<\omega,\alpha_1\>-l\not\equiv0\pmod p$.
Notice also that $0<k<p$. Indeed, $k\ge p$ implies $b\ge a_2$ and
$a_1+a_2\ge p+b\ge p+a_2$. Hence $a_1\ge p$, which is a contradiction.
Since the Weyl module $\Delta(\bar\omega)=\Delta(a_2\bar\omega_2)$ is simple,
Theorem~A(i) (where $m=0$) shows that the $G^{(1)}$-submodule of $L(\omega)$
generated by $X_{-\alpha_1,k}v^+_\omega$ is isomorphic to $\Delta((p+b)\bar\omega_2)$.
The latter module is already not simple. For example,
$X_{-\alpha_2,b+1}e^+_{(p+b)\bar\omega_2}$ is a nonzero $G^{(1)}$-primitive vector.
Thus $X_{-\alpha_2,b+1}X_{-\alpha_1,k}v^+_\omega$ is
a nonzero $G^{(1)}$-primitive vector of $L(\omega)$
of weight $\omega-(p+b-a_2)\alpha_1-(b+1)\alpha_2$.

There is an interesting connection  between this example and~\cite[Lemma 2.55]{Suprunenko_MinPol},
which is extensively used in that paper for calculation of
degrees of minimal polynomials.
In our notation,~\cite[Lemma~2.55]{Suprunenko_MinPol} is as follows:

\medskip
{
\noindent
\it
Let $M$ be an indecomposable $G^{(1)}$-module with highest
weight $(p+b)\bar\omega_2$ and $0\le b<p-1$.
Suppose that $X_{-\alpha_2,b+1}v^+\ne0$,
where $v^+$ is a highest weight vector of $M$.
Then $M\cong\Delta((p+b)\bar\omega_2)$.}
\medskip

\noindent
Therefore, if we somehow prove
that $X_{-\alpha_2,b+1}X_{-\alpha_1,k}v^+_\omega\ne0$,
then it will follow from this lemma that
the $G^{(1)}$-submodule of $L(\omega)$ generated
by $X_{-\alpha_1,k}v^+_\omega$ is a Weyl module
(without applying Theorem~A).

{\bf Example II.} Let $p=5$, $G=\SL_5(\F)$ and
$\omega=3\omega_1+3\omega_2+\omega_3+2\omega_4$. Take any
$k=1,\ldots,4$ and apply Theorem~A(i) for this $k$. The value
$k=4$ does not fit, since $\<\omega,\alpha_1\>-3=0$.

If we apply Theorem~A(i) for $k=1$, then we obtain that $L(\omega)$
contains a $G^{(1)}$-submodule isomorphic to $\Delta(\bar\omega+\bar\omega_2)$.
However, the last module is simple and we do not get any nonzero
$G^{(1)}$-primitive vectors in this way except the trivial $X_{-\alpha_1}v^+_\omega$.

The cases $k=2$ and $k=3$ on the contrary give new vectors.
In the former case, Theorem~A(i) implies that
$L(\omega)$ contains a $G^{(1)}$-submodule isomorphic
to $\Delta(\bar\omega+2\bar\omega_2)$. The last module contains nonzero
primitive vectors of weights $\bar\omega+2\bar\omega_2-\bar\alpha_2$ and
$\bar\omega+2\bar\omega_2-\bar\alpha_2-\bar\alpha_3-\bar\alpha_4$
by the Carter--Payne theorem (\cite{Carter_Payne}).
In the latter case, Theorem~A(i) implies that
$L(\omega)$ contains a $G^{(1)}$-submodule isomorphic
to $\Delta(\bar\omega+3\bar\omega_2)$. The last module contains nonzero
primitive vectors of weights $\bar\omega+3\bar\omega_2-2\bar\alpha_2$ and
$\bar\omega+3\bar\omega_2-2\bar\alpha_2-2\bar\alpha_3-2\bar\alpha_4$
by the Carter--Payne theorem (\cite{Carter_Payne}).

Thus except trivial nonzero $G^{(1)}$-primitive vectors of weights
$\omega-i\alpha_1$ with $i=0,\ldots,3$, the module $L(\omega)$
(which is not a Weyl module)
also contains nonzero $G^{(1)}$-primitive vectors of weights
$\omega-2\alpha_1-\alpha_2$, $\omega-2\alpha_1-\alpha_2-\alpha_3-\alpha_4$,
$\omega-3\alpha_1-2\alpha_2$ and $\omega-3\alpha_1-2\alpha_2-2\alpha_3-2\alpha_4$.

Computer calculations show that examples similar to Example~II are
quite abundant. Note that in both Examples~I and~II, we apply $X_{-\alpha_1,k}$ to $v^+_\omega$
only for $k>0$. The reason is that the case $k=0$ corresponds to
Smith's theorem (\cite{Smith1}) and the only primitive vectors of $L{\downarrow}_{G^{(q)}}$
produced in this way are those proportional to $v^+_\omega$.

We shall use the following result following directly from~\cite[Theorem~2]{Steinberg_eng}.
\vspace{-15pt}
\begin{proposition}\label{proposition:1}${}$
\!\!\!
\!\!
\!\!
\!\!
The products
$
\displaystyle
\prod\nolimits_{\alpha\in\Phi^+}\!X_{-\alpha,m_{-\alpha}}
\cdot
\prod\nolimits_{i=1}^{n{-}1}\!\!H_{\alpha_i,n_i}
\cdot
\prod\nolimits_{\alpha\in\Phi^+}\!X_{\alpha,m_\alpha},
$\\[4pt]
where $m_{-\alpha}$, $n_i$, $m_\alpha\in\Z^+$,
taken in any fixed order
form a basis of $\U$.
\end{proposition}
We denote by $\U^+$ the subspace of $\U$ spanned by the above products with unitary first and second factors.
Given integers $q_1,\ldots,q_m\in\{1,\ldots,n-1\}$, we denote by $\U^{(q_1,\ldots,q_m)}$ the subspace of $\U$
spanned by all the above products such that $m_\alpha=0$ unless $\alpha\in\Phi^{(q_1,\ldots,q_m)}$
and $n_i=0$ unless $i\in\{1,\ldots,n-1\}\setminus\{q_1,\ldots,q_m\}$.
One can easily see that $\U^+$ and $\U^{(q_1,\ldots,q_m)}$ are subalgebras of $\U$.
We let $\U^{(q_1,\ldots,q_m)}$ act on any rational $G^{(q_1,\ldots,q_m)}$-module according to~(\ref{equation:-1}).
In~the sequel, we shall mean the $X(T)$-grading of $\U$
in which $X_{\alpha,m}$ has weight $m\alpha$ and
$H_{\alpha_i,m}$ has weight $0$.

For each $\omega\in X^+(T)$, we denote by $\nabla(\omega)$
the module contravariantly dual to the Weyl module $\Delta(\omega)$ and
denote by $\pi^\omega:\Delta(\omega)\to L(\omega)$ the $G$-module
epimorphism such that $\pi^\omega(e^+_\omega)=v^+_\omega$.
We also denote by $V^\tau$ for $\tau\in X(T)$
the $\tau$-weight space of a rational $T$-module $V$.

A vector $v$ of a rational $G$-module is called {\it $G$-primitive} if $\F v$ is fixed by the Borel subgroup $B$.
We use similar terminology for $G^{(q_1,\ldots,q_m)}$ and omit the prefix when it is clear which group we mean.
In view of the universal property of Weyl modules~\cite[Lemma~II.2.13 b]{Jantzen2},
we can speak about primitive vectors of a rational module $V$ instead of homomorphisms from Weyl modules to $V$
(we use this language in Theorem~B).

Note that Theorems~A and~B in the case $q=n-1$ are easy consequences of
the theorems in the case $q=1$ by a standard argument involving
twisting with the automorphism $g \mapsto w_0(g^{-1})^{\rm t}w_0^{-1}$, where ${}^{\rm t}$ stands for the transposition
and $w_0$ stands for for the longest element of the Weyl group.
Therefore in the remainder of the article we consider only the case $q=1$.

{\bf Acknowledgments.} The author would like to thank Irina
Suprunenko for drawing his attention to this problem and helpful
discussions.

\section{Proof of the main results}\label{Proof of the main results}

We fix a weight $\omega=a_1\omega_1+\cdots+a_{n-1}\omega_{n-1}$ of $X^+(T)$
and an integer $k\in\Z^+$. The restriction of $\omega$ to $T^{(1)}$ is
$\bar\omega=a_2\bar\omega_2+\cdots+a_{n-1}\bar\omega_{n-1}$.
Clearly, $X_{-\alpha_1,k}v^+_\omega$ is a (possibly zero) $G^{(1)}$-primitive vector
of $T^{(1)}$-weight $\bar\omega+k\bar\omega_2$.
By the universal property of Weyl modules~\cite[Lemma~II.2.13 b]{Jantzen2},
there exists the homomorphism
$\phi_k^\omega:\Delta(\bar\omega+k\bar\omega_2)\to L(\omega)$
of $G^{(1)}$-modules that takes
$e^+_{\bar\omega+k\bar\omega_2}$ to $X_{-\alpha_1,k}v^+_\omega$.
Obviously,
\begin{equation}\label{equation:0}
KG^{(1)}X_{-\alpha_1,k}v^+_\omega\cong\Delta(\bar\omega+k\bar\omega_2)/\ker\phi^\omega_k.
\end{equation}

Problem~\ref{problem:1} can now be reformulated as follows: {\it Describe the weights $\omega\in X^+(T)$
such that $\ker\phi^\omega_k=0$.}
The analog of this problem for $\Delta(\omega)$ has a trivial solution.
\begin{lemma}\label{lemma:0}
The $G^{(1)}$-submodule of $\Delta(\omega)$ generated
by $X_{-\alpha_1,k}e^+_\omega$ is isomorphic to
$\Delta(\bar\omega+k\bar\omega_2)$
if
$0\le k\le\<\omega,\alpha_1\>$
and is zero otherwise.
\end{lemma}
\begin{proof} Suppose temporarily that $\ch\F=0$. Then $\Delta(\omega)$ is irreducible.
Since $X_{\alpha_1,k}X_{-\alpha_1,k}e^+_\omega=\binom{a_1}ke^+_\omega$,
we have (recall that $\alpha_1$ is simple)
\begin{equation}\label{equation:1}
\dim\Delta(\omega)^{\omega-k\alpha_1}=
\left\{
\begin{array}{ll}
1&\mbox{ if } 0\le k\le a_1;\\[3pt]
0&\mbox{ otherwise}.
\end{array}
\right.
\end{equation}
Now let us return to the situation where $\ch\F=p>0$. Since the character of a Weyl module
does not depend on $\ch\F$,~(\ref{equation:1}) holds again.
Therefore, $X_{-\alpha_1,k}e^+_\omega=0$ if $k>a_1$.
Thus we assume $0\le k\le a_1$ for the rest of the proof.
Consider the decomposition
$\Delta(\omega)=\bigoplus_{b\in\Z^+}V^{(b)}$, where
$$
V^{(b)}=\bigoplus\nolimits_{b_2,\ldots,b_{n-1}\in\Z^+}\Delta(\omega)^{\omega-b\alpha_1-b_2\alpha_2-\cdots-b_{n-1}\alpha_{n-1}}
$$
(the $b$th level of $\Delta(\omega)$). Note that each $V^{(b)}$
is a $G^{(1)}$-module. By~(\ref{equation:1}), $X_{-\alpha_1,k}e^+_\omega$
is a nonzero vector of $V^{(k)}$ having $T^{(1)}$-weight
$\bar\omega+k\bar\omega_2$.
Moreover, the weight space of $V^{(k)}$ corresponding to this weight is one-dimensional.
Any other $T^{(1)}$-weight of $V^{(k)}$ is less than this weight.
It follows from~\cite{Mathieu2}
(see also~\cite[Proposition II.4.24]{Jantzen2}) that $\Delta(\omega){\downarrow}_{G^{(1)}}$
has a Weyl filtration. By~\cite[Proposition II.4.16(iii)]{Jantzen2},
its direct summand $V^{(k)}$ also has a Weyl filtration (as a $G^{(1)}$-module).
Any such filtration contains one factor isomorphic to $\Delta(\bar\omega+k\bar\omega_2)$
and, possibly, some other factors each isomorphic to $\Delta(\tau)$ with $\tau<\bar\omega+k\bar\omega_2$.
Applying~\cite[II.4.16~Remark~4]{Jantzen2} to the dual module ${V^{(k)}}^*$,
we obtain that $V^{(k)}$ contains a $G^{(1)}$-submodule isomorphic to
$\Delta(\bar\omega+k\bar\omega_2)$. Clearly, this submodule is generated
by $X_{-\alpha_1,k}e^+_\omega$.
\end{proof}

We deliberately did not use a basis of $\Delta(\omega)$
in the proof of the above theorem to make it valid
for $G$ of arbitrary type.

\begin{lemma}\label{lemma:1}
The modules $KG^{(1)}X_{-\alpha_1,k}v^+_\omega$ and $KG^{(1)}X_{-\alpha_1,k}e^+_\omega$ decompose
into direct sums of their $T$-weight subspaces.
These sums are exactly the decompositions into $T^{(1)}$-weight subspaces.
\end{lemma}
\begin{proof}
The only fact we need to prove is that
$\overline{\omega-b_1\alpha_1-\cdots-b_{n-1}\alpha_{n-1}}=\overline{\omega-c_1\alpha_1-\cdots-c_{n-1}\alpha_{n-1}\vphantom{b}}$
and $b_1=c_1$ imply $b_i=c_i$ for any $i{=}1,{\ldots},n{-}1$. This is obvious, since the first equality is equivalent~to
$\bar\omega+b_1\bar\omega_2-b_2\bar\alpha_2-{\cdots}-b_{n-1}\bar\alpha_{n-1}=
\bar\omega+c_1\bar\omega_2-c_2\bar\alpha_2-{\cdots}-c_{n-1}\bar\alpha_{n-1}$.
\end{proof}

Before proving Theorem~B, we need to describe the standard bases for Weyl modules over $G^{(1)}$.
Let $\kappa=d_2\bar\omega_2+\cdots+d_{n-1}\bar\omega_{n-1}$ be a weight of $X^+(T^{(1)})$.
A sequence $\lambda=(\lambda_2,\ldots,\lambda_n)$ of nonnegative integers is called {\it coherent} with $\kappa$ if
$d_i=\lambda_i-\lambda_{i+1}$ for any $i=2,\ldots,n-1$. The {\it diagram} of $\lambda$ is the set
$$
[\lambda]=\{(i,j)\in\Z^2\|2\le i\le n\mbox{ and }1\le j\le \lambda_i\}.
$$
We shall think of $[\lambda]$ as an array of boxes.
For example, if $\lambda{=}(5,3,2,0)$ then

{
\begin{center}
\setlength{\unitlength}{0.4mm}
\begin{picture}(45,0)
\put(0,0){\line(1,0){50}}
\put(0,-10){\line(1,0){50}}
\put(0,-20){\line(1,0){30}}
\put(0,-30){\line(1,0){20}}
\put(0,0){\line(0,-1){30}}
\put(10,0){\line(0,-1){30}}
\put(20,0){\line(0,-1){30}}
\put(30,0){\line(0,-1){20}}
\put(40,0){\line(0,-1){10}}
\put(50,0){\line(0,-1){10}}
\put(-25,-17){$[\lambda]=$}
\end{picture}
\end{center}
}

\vspace{1.3cm}
\noindent
Note that in our terminology the top row of this diagram is the second row.

A {\it $\lambda$-tableau} is a function $t:[\lambda]\to\{2,\ldots,n\}$,
which we regard as the diagram $[\lambda]$ filled with
integers in $\{2,\ldots,n\}$.
A $\lambda$-tableau $t$ is called {\it row standard} if its entries weakly
increase along the rows, that is $t(i,j)\le t(i,j')$ if $j<j'$.
A $\lambda$-tableau $t$ is called {\it regular row standard} if it is
row standard and every entry in row $i$ of $t$ is at least $i$.
Finally, a $\lambda$-tableau $t$
is called {\it standard} if it is row standard
and its entries strictly increase down the columns,
that is $t(i,j)<t(i',j)$ if $i<i'$. For example,

{
\begin{center}
\setlength{\unitlength}{0.4mm}
\begin{picture}(45,0)
\put(0,0){\line(1,0){50}}
\put(0,-10){\line(1,0){50}}
\put(0,-20){\line(1,0){30}}
\put(0,-30){\line(1,0){20}}
\put(0,0){\line(0,-1){30}}
\put(10,0){\line(0,-1){30}}
\put(20,0){\line(0,-1){30}}
\put(30,0){\line(0,-1){20}}
\put(40,0){\line(0,-1){10}}
\put(50,0){\line(0,-1){10}}
\put(-21,-17){$t=$}
\put(3,-8.5){2} \put(13,-8.5){3} \put(23,-8.5){3} \put(33,-8.5){4}\put(43,-8.5){5}
\put(3,-18.5){3}\put(13,-18.5){4}\put(23,-18.5){4}
\put(3,-28.5){4}\put(13,-28.5){5}
\end{picture}
\end{center}
}

\vspace{1.3cm}

\noindent
is a standard $(5,3,2,0)$-tableau.
For any $\lambda$-tableau $t$, we put
$$
F_t:=\prod\nolimits_{2\le a<b\le n}X_{-\alpha_a-\cdots-\alpha_{b-1},N_{a,b}},
$$
where $N_{a,b}$ is the number of entries $b$ in row $a$ of $t$,
$X_{-\alpha_a-\cdots-\alpha_{b-1},N_{a,b}}$
precedes
$X_{-\alpha_c-\cdots-\alpha_{d-1},N_{c,d}}$
if $b<d$ or $b=d$ and $a<c$.

\begin{remark}\label{remark:2} One can easily see that the number of entries greater than $2$ in the second (top) row of $t$
is exactly minus the coefficient at $\alpha_2$ in the weight of $F_t$.
\end{remark}

For $t$ as in the above example, we have
$$
F_t=X_{-\alpha_2,2}\,X_{-\alpha_2-\alpha_3}\,X_{-\alpha_3,2}\,X_{-\alpha_2-\alpha_3-\alpha_4}\,X_{-\alpha_4}.
$$

\begin{proposition}[\cite{Carter_Lusztig}]\label{basis of Weyl module}
Let $\kappa$ be a weight of $X^+(T^{(1)})$ and $\lambda=(\lambda_2,\ldots,\lambda_n)$ be
a sequence coherent with $\kappa$.
Then the vectors $F_te^+_\kappa$, where $t$ is a standard $\lambda$-tableau,
form a basis of $\Delta(\kappa)$.
\end{proposition}

Now suppose that $m=3,\ldots,n$ and $\lambda_2-\lambda_3=d_2\ge1$.
For any regular row standard $\lambda$-tableau $t$,
we define $\rho_m(t)$ to be the $(\lambda_2-1,\lambda_3,\ldots,\lambda_n)$-tableau obtained from $t$ by removing one entry
$m$ from the second row, if such removal is possible, and shifting all elements of the resulting row to the left.

One can easily check that for any $2{\le}s{<}m{\le}n$ and $N\in\Z^+$, there holds
\begin{equation}\label{equation:3}
[X_{\alpha_1+\cdots+\alpha_{m-1}},X_{-\alpha_s-\cdots-\alpha_{m-1},N}]=X_{-\alpha_s-\cdots-\alpha_{m-1},N-1}X_{\alpha_1+\cdots+\alpha_{s-1}}.
\end{equation}
Note that~(\ref{equation:3}) holds for any $N\in\Z$ if we define $X_{\alpha,N}:=0$ for $N<0$. Let $I^+$
denote the left ideal of $\U$ generated by the elements $X_{\alpha,N}$ with $\alpha\in\Phi^+$ and $N>0$.

\begin{lemma}\label{lemma:2}
Let $m=3,\ldots,n$, $\lambda_2-\lambda_3=d_2\ge1$, $t$ be a regular row standard $\lambda$-tableau and $1\le k$.
We have
$$
X_{\alpha_1+\cdots+\alpha_{m-1}}F_tX_{-\alpha_1,k}\=F_{\rho_m(t)}X_{-\alpha_1,k-1}(H_{\alpha_1}+1-k)\pmod{I^+}
$$
if $\rho_m(t)$ is well-defined and
$$
X_{\alpha_1+\cdots+\alpha_{m-1}}F_tX_{-\alpha_1,k}\=0\pmod{I^+}
$$
otherwise.
\end{lemma}
\begin{proof} Let $N_{a,b}$ denote the number of entries $b$ in row $a$ of $t$.
Consider the representation $F_t=F_3\cdots F_n$, where
$$
F_j=X_{-\alpha_2-\cdots-\alpha_{j-1},N_{2,j}}\cdots X_{-\alpha_{j-2}-\alpha_{j-1},N_{j-2,j}}X_{-\alpha_{j-1},N_{j-1,j}}.
$$
Clearly, $X_{\alpha_1+\cdots+\alpha_{m-1}}$ commutes with any $F_j$ such that $j\ne m$. Using (\ref{equation:3})
and the fact that $X_{\alpha_1+\cdots+\alpha_{s-1}}$ commutes with any factor of $F_m$ for $s=2,\ldots,m-1$, we obtain
$$
{\arraycolsep=0pt
\begin{array}{l}
X_{\alpha_1+\cdots+\alpha_{m-1}}F_m=F_mX_{\alpha_1+\cdots+\alpha_{m-1}}+\\[12pt]
\displaystyle\sum\nolimits_{s=2}^{m-1}\biggl(\prod\nolimits_{l=2}^{m-1}X_{-\alpha_l-\cdots-\alpha_{m-1},N_{l,m}-\delta_{l,s}}\!\biggr)X_{\alpha_1+\cdots+\alpha_{s-1}}.
\end{array}}
$$
Here and in what follows $\delta_{l,s}$ equals $1$ if $l=s$ and equals $0$ otherwise.
Since $X_{\alpha_1+\cdots+\alpha_{s-1}}$ commutes with any $F_j$ for $s=2,\ldots,m$ and $j=$\linebreak $m+1,\ldots,n$,
we obtain
$$
{\arraycolsep=0pt
\begin{array}{l}
X_{\alpha_1+\cdots+\alpha_{m-1}}F_tX_{-\alpha_1,k}=F_tX_{\alpha_1+\cdots+\alpha_{m-1}}X_{-\alpha_1,k}+\\[12pt]
\displaystyle\sum_{s=2}^{m-1}F_1{\cdots}F_{m-1}\Biggl(\prod_{l=2}^{m-1}\!X_{-\alpha_l-\cdots-\alpha_{m-1},N_{l,m}-\delta_{l,s}}\!\Biggr)F_{m+1}{\cdots}F_nX_{\alpha_1+\cdots+\alpha_{s-1}}X_{-\alpha_1,k}.
\end{array}}
$$
Since $m\ge3$ the first summand and any product under the summation sign for $s>2$ in the right-hand side
of the above formula belongs to $I^+$. Hence
$$
{\arraycolsep=0pt
\begin{array}{l}
X_{\alpha_1+\cdots+\alpha_{m-1}}F_tX_{-\alpha_1,k}\=
\displaystyle F_1{\cdots}F_{m-1}\biggl(\prod\nolimits_{l=2}^{m-1}\!X_{-\alpha_l-\cdots-\alpha_{m-1},N_{l,m}-\delta_{l,2}}\!\biggr)\times\\[12pt]
\times F_{m+1}{\cdots}F_nX_{-\alpha_1,k-1}(H_{\alpha_1}{-}k+1)\pmod{I^+}.
\end{array}}
$$
If $N_{2,m}>0$ then the right-hand side of the above formula equals \linebreak $F_{\rho_m(t)}X_{-\alpha_1,k-1}(H_{\alpha_1}+1-k)$.
Otherwise it equals zero and $\rho_m(t)$ is not well-defined.
\end{proof}

We also need the iterated version of $\rho_m$. Suppose that $M=(m_1,\ldots,m_l)$ is a sequence with entries in
$\{3,\ldots,n\}$ and $\lambda_2-\lambda_3=d_2\ge l$. For any regular row standard $\lambda$-tableau $t$,
we define $\rho_M(t)$ to be the $(\lambda_2-l,\lambda_3,\ldots,\lambda_n)$-tableau obtained from $t$ by removing the entries
$m_1,\ldots,m_l$ (taking into account their multiplicities) from the second row, if such removal is possible,
and shifting all elements of the resulting row to the left.
We clearly have $\rho_M(t)=\rho_{m_1}\circ\cdots\circ\rho_{m_l}(t)$ if the second row of $t$ contains entries $m_1,\ldots,m_l$.
Hence applying Lemma~\ref{lemma:2}, we obtain the following result.

\begin{corollary}\label{corollary:1}
Let $M{=}(m_1,{\ldots},m_l)$ be a sequence with entries in
$\{3,{\ldots},n\}$, $\lambda_2-\lambda_3=d_2\ge l$, $t$ be a regular row standard $\lambda$-tableau and $l\le k$.
We have
$$
\(\prod_{i=1}^lX_{\alpha_1+\cdots+\alpha_{m_i-1}}\!\)\!F_tX_{-\alpha_1,k}\=F_{\rho_M(t)}X_{-\alpha_1,k-l}\!\(\prod_{i=1}^lH_{\alpha_1}+i-k\)\!\!\!\!\!\!\!\pmod{I^+}
$$
if $\rho_M(t)$ is well-defined and
$$
\(\prod_{i=1}^lX_{\alpha_1+\cdots+\alpha_{m_i-1}}\!\)\!F_tX_{-\alpha_1,k}\=0\pmod{I^+}
$$
otherwise.
\end{corollary}

In what follows, $\coeff_{\alpha_1}(\beta)$ denotes
the coefficient at $\alpha_1$ of a root $\beta\in\Phi$.

\begin{proof}[Proof of Theorem~B]
{\it ``Only if part''.} Suppose that the $G^{(1)}$-submodule of $L(\omega)$
generated by $X_{-\alpha_1,k}v^+_\omega$ is isomorphic to a Weyl module.
Then \linebreak$X_{-\alpha_1,k}v^+_\omega\ne0$ and
$X_{\alpha_1,k}X_{-\alpha_1,k}v^+_\omega=\binom{a_1}kv^+_\omega\ne0$.
Hence $a_1-l\not\equiv0\pmod p$ for $l=0,\ldots,k-1$, since $k<p$.

Now let $v$ be a nonzero $G^{(1)}$-primitive vector of
$KG^{(1)}X_{-\alpha_1,k}v^+_\omega$. By Lemma~\ref{lemma:1},
$v$ is a $T$-weight vector. It has $T$-weight $\omega-\delta$,
where $\delta$ is a sum of positive roots.
Clearly, the coefficient at $\alpha_1$ of $\delta$ equals $k$.
We claim that

\begin{equation}\label{equation:3.5}
\arraycolsep=1pt
\begin{array}{rcl}
\delta\in E(1,k)&:=&\{\beta_1+\cdots+\beta_l\|\beta_1,\ldots,\beta_l\in\Phi^+,\coeff_{\alpha_1}(\beta_1){>}0,\ldots,\\[6pt]
      &  &\,\,\,\coeff_{\alpha_1}(\beta_l){>}0,\coeff_{\alpha_1}(\beta_1)+\cdots+\coeff_{\alpha_1}(\beta_l)=k\}.
\end{array}
\end{equation}

Indeed, By Proposition~\ref{proposition:1}, the products
$\prod\nolimits_{\alpha\in\Phi^+}X_{\alpha,m_\alpha}$
taken in any fixed order form a basis of $\U^+$.
Let us assume now that this order is such that
any factor $X_{\alpha,m_\alpha}$ with $\coeff_{\alpha_1}(\alpha)>0$
is situated to the left of any factor $X_{\beta,m_\beta}$ with
$\coeff_{\alpha_1}(\beta)=0$.
Since $v\ne0$, we have $\(\prod\nolimits_{\alpha\in\Phi^+}X_{\alpha,m_\alpha}\)v=cv^+_\omega$
for some $c\in\F^*$ and $m_\alpha\in\Z^+$ such that $\sum_{\alpha\in\Phi^+}m_\alpha\alpha=\delta$.
Since $v$ is $G^{(1)}$-primitive, the order of factors we have chosen implies that $m_\alpha=0$ if
$\coeff_{\alpha_1}(\alpha)=0$. On the other hand,
$\sum_{\alpha\in\Phi^+}m_\alpha\coeff_{\alpha_1}(\alpha)=k$.
Hence~(\ref{equation:3.5}) directly follows.

Now it remains to notice that
$$
E(1,k)  =\{b_1\alpha_1+\cdots+b_{n-1}\alpha_{n-1}\|k=b_1\ge b_2\ge\cdots\ge b_{n-1}\ge0\}.
$$

{\it ``If part''.}
We assume that
$a_1-l\not\equiv0\pmod p$ for any $l=0,\ldots,k-1$ and
any nonzero primitive vector of $\Delta(\bar\omega+k\bar\omega_2)$ has weight as in~\ref{condition:case:q=1}.
In particular, we have $k\le a_1$.
Suppose that $KG^{(1)}X_{-\alpha_1,k}v^+_\omega$ is not isomorphic to a Weyl module. Then by~(\ref{equation:0}),
we get $\ker\phi_k^\omega\ne0$. Since $\ker\phi_k^\omega$ is a submodule of
$\Delta(\bar\omega+k\bar\omega_2)$, it contains a nonzero primitive vector $u$.
Our assumption implies that $u$ has weight
$\bar\omega+k\bar\omega_2-b_2\bar\alpha_2-\cdots-b_{n-1}\bar\alpha_{n-1}$, where $k\ge b_2\ge\cdots\ge b_{n-1}\ge0$.

The universal property of Weyl modules implies the existence of
the $G^{(1)}$-module homomorphism
$\gamma:\Delta(\bar\omega+k\bar\omega_2)\to KG^{(1)}X_{-\alpha_1,k}e^+_\omega$ such that
$\gamma(e^+_{\bar\omega+k\bar\omega_2})=X_{-\alpha_1,k}e^+_\omega$.
Lemma~\ref{lemma:0} shows that
$\gamma$ is an isomorphism. Since $\pi^\omega\circ\gamma=\phi_k^\omega$
(to prove it, apply both sides to $e^+_{\bar\omega+k\bar\omega_2}$),
we have $\gamma(u)\in\rad\Delta(\omega)$.

Take any sequence $\lambda=(\lambda_2,\ldots,\lambda_n)$ of
nonnegative integers coherent with $\bar\omega+k\bar\omega_2$. In particular,
we have $\lambda_2-\lambda_3=\<\bar\omega+k\bar\omega_2,\bar\alpha_2\>=a_2+k\ge b_2$.
By Proposition~\ref{basis of Weyl module}, we have the
representation $u{=}\sum_{s\in S}c_sF_se^+_{\bar\omega+k\bar\omega_2}$, where $c_s\in\F^*$ and
$S$ is a nonempty set consisting of standard $\lambda$-tableaux
$s$ such that $F_s$ has weight
$-b_2\alpha_2-\cdots-b_{n-1}\alpha_{n-1}$. Obviously, any tableau
$s\in S$ has exactly $b_2$ entries greater than $2$ in the second
row (see Remark~\ref{remark:2}).

Let us fix some tableau $t\in S$, denote be $m_1,\ldots,m_{b_2}$ all the entries
greater than $2$ in the second row of $t$ (taking into account multiplicities) and put $M:=(m_1,\ldots,m_{b_2})$.
Clearly, $\rho_M(t)$ is well-defined. Moreover, for any $s\in S$ such that $\rho_M(s)$ is well-defined,
$\rho_M(s)$ is a standard $(\lambda_2-b_2,\lambda_3,\ldots,\lambda_n)$-tableau whose every entry in the second row is $2$ and
$F_{\rho_M(s)}$ has weight
$$
-b_2\alpha_2-\cdots-b_{n-1}\alpha_{n-1}+\(\!\sum\nolimits_{i=1}^{b_2}\alpha_2+\cdots+\alpha_{m_i-1}\!\)=-b'_3\alpha_3-\cdots-b'_{n-1}\alpha_{n-1},
$$
where $b'_3,\ldots,b'_{n-1}$ are nonnegative integers (independent of $s$).
Applying $\gamma$ to the above representation of $u$, we obtain
$$
\gamma(u)=\sum\nolimits_{s\in S}c_sF_sX_{-\alpha_1,k}e^+_\omega\in\rad\Delta(\omega).
$$
Multiplying this formula by $\(\prod_{i=1}^{b_2}X_{\alpha_1+\cdots+\alpha_{m_i-1}}\)$ on the left,
taking into account $b_2\le k$ and applying Corollary~\ref{corollary:1}, we obtain
\begin{equation}\label{equation:4}
{\arraycolsep=-2pt
\begin{array}{l}
\displaystyle\(\prod_{i=1}^{b_2}a_1{+}i{-}k\!\)\!
\sum\Bigl\{F_{\rho_M(s)}X_{-\alpha_1,k-b_2}e^+_\omega\Bigl|s{\in}S\mbox{ and }\rho_M(s)\mbox{ is well-defined}\Bigr\}\\[16pt]
\in\rad\Delta(\omega).
\end{array}}
\end{equation}
Since $b_2\le k$ and we assumed $a_1-l\not\equiv0\pmod p$ for any $l=0,\ldots,k-1$,
the fist factor of the product in the left-hand side of the above formula is nonzero. Moreover, if $s$ and $s'$ are distinct tableaux of $S$ and
both $\rho_M(s)$ and $\rho_M(s')$ are well-defined, then $\rho_M(s)\ne\rho_M(s')$. Notice that the summation in~(\ref{equation:4}) is nonempty,
since at least $s=t$ satisfies the restrictions.

By Lemma~\ref{lemma:0}, the $G^{(1)}$-submodule $W$ of $\Delta(\omega)$ generated by $X_{-\alpha_1,k-b_2}e^+_\omega$
is isomorphic to $\Delta(\bar\omega+(k-b_2)\bar\omega_2)$. Note that $(\lambda_2-b_2,\lambda_3,\ldots,\lambda_n)$ is coherent
with $\bar\omega+(k-b_2)\bar\omega_2$. Therefore by Proposition~\ref{basis of Weyl module},
the left-hand side of~(\ref{equation:4}) is nonzero.
It belongs to a proper
$G^{(1)}$-submodule $W\cap\rad\Delta(\omega)$ of $W$ and hence to $\rad W$.
Note that $X_{\alpha_1,k-b_2}X_{-\alpha_1,k-b_2}e^+_\omega=\binom{a_1}{k-b_2}e^+_\omega\ne0$,
whence $X_{-\alpha_1,k-b_2}e^+_\omega\notin\rad\Delta(\omega)$ and indeed $W\cap\rad\Delta(\omega)\ne W$.

In other words, we proved that $\rad\Delta(\bar\omega+(k-b_2)\bar\omega_2)$ contains a nonzero vector $u'$
of weight $\bar\omega+(k-b_2)\bar\omega_2-b'_3\bar\alpha_3-\cdots-b'_{n-1}\bar\alpha_{n-1}$.
As an immediate consequence of this fact, we get $n\ge4$.
For any weight $\kappa\in X(T)$, we denote by $\widetilde\kappa$ its restriction~to~$T^{(1,2)}$.
By Lemma~\ref{lemma:0}, the $G^{(1,2)}$-submodule $W'$ of $\Delta(\bar\omega{+}(k{-}b_2)\bar\omega_2)$
generated by $e^+_{\bar\omega{+}(k{-}b_2)\bar\omega_2}$ is isomorphic to $\Delta(\widetilde\omega)$
(the restriction of ${\bar\omega+(k-b_2)\bar\omega_2}$ to $T^{(1,2)}$ is $\widetilde\omega$).
Clearly, $u'$ belongs to a proper submodule $W'\cap\rad\Delta(\bar\omega+(k-b_2)\bar\omega_2)$ of $W'$
and thus belongs to $\rad W'$. In  this way, we proved that $\Delta(\widetilde\omega)$ is not simple.

Consider the $G^{(1,2)}$-submodule $W''$ of $\Delta(\bar\omega+k\bar\omega_2)$ generated by
$e^+_{\bar\omega+k\bar\omega_2}$. By Lemma~\ref{lemma:0}, $W''$ is isomorphic to $\Delta(\widetilde\omega)$
(the restriction of $\bar\omega+k\bar\omega_2$ to $T^{(1,2)}$ is also $\widetilde\omega$).
Therefore $W''$ is not simple and contains a nonzero $G^{(1,2)}$-primitive vector $u''$
of $T^{(1,2)}$-weight $\widetilde\omega-d_3\widetilde\alpha_3-\cdots-d_{n-1}\widetilde\alpha_{n-1}$,
where $d_3,\ldots,d_{n-1}$ are nonnegative integers not equal simultaneously to zero.
By Lemma~\ref{lemma:1}, we obtain that $u''$ has $T^{(1)}$-weight
$\bar\omega+k\bar\omega_2-d_3\bar\alpha_3-\cdots-d_{n-1}\bar\alpha_{n-1}$.
Note that this weight does not have the form described in~\ref{condition:case:q=1}.
Since $x_{\alpha_2}(t)$ commutes with any $x_{-\alpha_i}(s)$, where $i=3,\ldots,n{-}1$, and
$$
u''\in W''=\F\,\<x_{-\alpha_i}(s)\|i=3,\ldots,n-1,s\in\F\>\,e_{\bar\omega+k\bar\omega_2},
$$
we obtain that $u''$ is $G^{(1)}$-primitive. This is a contradiction.
\end{proof}

\begin{proof}[Proof of Theorem~A] Suppose that the hypothesis of~\ref{condition:G1} holds.
The weights of $\Delta(\bar\omega_2)$ are $\kappa_1,\ldots,\kappa_{n-1}$,
where $\kappa_i=\bar\omega_2-\bar\alpha_2-\cdots-\bar\alpha_i$ and each weight space is one-dimensional.

Suppose for a while that $\ch\F=0$. It is well known that for any $\kappa\in X^+(T^{(1)})$,
the module $\Delta(\kappa)\otimes\Delta(\bar\omega_2)$
is a direct sum of
$\Delta(\kappa+\kappa_i)$ over $i=1,\ldots,n-1$ such that $\kappa+\kappa_i\in X^+(T^{(1)})$
(see, for example,~\cite[Lemma~4.8]{Kleshchev_tf3}).
Thus the module $\Delta(\bar\omega+m\bar\omega_2)\otimes\Delta(\bar\omega_2)^{\otimes k-m}$
is a direct sum of several copies of $\Delta(\bar\omega+m\bar\omega_2+\kappa_{i_1}+\cdots+\kappa_{i_{k-m}})$ over sequences
$i_1,\ldots,i_{k-m}$ of integers in $\{1,\ldots,n-1\}$ such that
$\bar\omega+m\bar\omega_2+\kappa_{i_1}+\cdots+\kappa_{i_{k-m}}\in X^+(T^{(1)})$. Moreover, the module
$\Delta(\bar\omega+k\bar\omega_2)$ enters into this sum with multiplicity one.

Let us return to the case $\ch\F=p>0$. Applying the main result of~\cite{Mathieu2}, we obtain that
the module $V:=\Delta(\bar\omega+m\bar\omega_2)\otimes\Delta(\bar\omega_2)^{\otimes k-m}$ has a filtration
with factors $\Delta(\bar\omega+m\bar\omega_2+\kappa_{i_1}+\cdots+\kappa_{i_{k-m}})$ over the same sequences
$i_1,\ldots,i_{k-m}$ with the same multiplicities.
By~\cite[II.4.16~Remark 4]{Jantzen2} applied to the dual module $V^*$,
$V$ has a submodule isomorphic to $\Delta(\bar\omega+k\bar\omega_2)$.

Now recall that $\Delta(\bar\omega{+}m\bar\omega_2)\cong\nabla(\bar\omega{+}m\bar\omega_2)$
by the hypothesis of the present lemma and $\Delta(\bar\omega_2)\cong\nabla(\bar\omega_2)$.
Therefore, $V$ is isomorphic to $\nabla(\bar\omega{+}m\bar\omega_2)\otimes\nabla(\bar\omega_2)^{\otimes k-m}$ and
by the main result of~\cite{Mathieu2}
has a filtration with factors \linebreak$\nabla(\bar\omega+m\bar\omega_2+\kappa_{i_1}+\cdots+\kappa_{i_{k-m}})$ over the same sequences
$i_1,\ldots,i_{k-m}$ with the same multiplicities. Applying~\cite[Proposition II.4.13]{Jantzen2}, we obtain that $\Hom_{G^{(1)}}(\Delta(\kappa),V)=0$
unless $\kappa=\bar\omega+m\bar\omega_2+\kappa_{i_1}+\cdots+\kappa_{i_{k-m}}$. Since $V$ has a submodule isomorphic
to $\Delta(\bar\omega+k\bar\omega_2)$, any nonzero primitive vector of $\Delta(\bar\omega+k\bar\omega_2)$ has weight
$\bar\omega+m\bar\omega_2+\kappa_{i_1}+\cdots+\kappa_{i_{k-m}}$ with $i_1,\ldots,i_{k-m}$ as above.
It remains to apply Theorem~B\ref{condition:case:q=1}.

Part~\ref{condition:Gn-1} can be proved similarly but tensoring with $\Delta(\bbar\omega_{n-2})$ and
applying Theorem~B\ref{condition:case:q=n-1}.
\end{proof}

\end{document}